\documentclass[reqno]{amsart}  

\usepackage{latexsym}
\usepackage{amsmath,amsthm}
\usepackage{amssymb}
\usepackage{graphicx}

\input diagxy
 \xyoption{curve}

\def\R{\mathbb{R}}

\theoremstyle{plain}
\newtheorem{theorem}{Theorem}[]
\newtheorem{lemma} [theorem]{Lemma}
\newtheorem{proposition} [theorem]{Proposition}

\newtheorem{corollary} [theorem]{Corollary}

\theoremstyle{definition}
\newtheorem{definition} [theorem]{Definition}

\newtheorem{remark}{Remark}

\begin{document}

\title{Hopf-Rinow Theorem of sub-Finslerian geometry}
\author{Layth M. Alabdulsada and L\'aszl\'o Kozma }

\address{Layth M. Alabdulsada,  Institute of Mathematics, University of Debrecen, H-4002 Debrecen, P.O. Box 400, Hungary}
\email{layth.muhsin@science.unideb.hu}
\address{L\'aszl\'o  Kozma, Institute of Mathematics,
 University of Debrecen, H-4002 Debrecen, P.O. Box 400, Hungary}
\email{kozma@unideb.hu}


\bibliographystyle{alpha}

\begin{abstract}
The sub-Finslerian geometry means that the metric $F$ is defined only on a given subbundle of the tangent bundle, called a horizontal bundle.
In the paper, a version of the Hopf-Rinow theorem is proved in the case of sub-Finslerian manifolds, which relates the properties of completeness, geodesically completeness, and compactness. The sub-Finsler bundle, the exponential map and the Legendre transformation are deeply involved in this investigation.
\end{abstract}
\maketitle

\section{Introduction}
In the Riemannian and Finslerian geometry, there are two concepts of completeness. The first is the completeness in the sense of metric spaces, using the Riemannian metric. Secondly, a Riemannian or Finsler manifold $M$  is called geodesically complete if any geodesic $\gamma(t)$ starting from $x \in M$ is defined for all values of $t \in  \mathbb{R}$. On the other hand, the completeness in the Finsler geometry is divided into forward and backward geodesically completenesses, according to forward and backward distance metrics,  resp.

Hopf-Rinow theorem is a basic theorem of complete Riemannian manifolds, which connects the completeness properties with compactness, and the exponential map. Its consequence says that any two points of a complete manifold can be connected by a length minimizing geodesic. In 1931, H. Hopf and W. Rinow showed their theorem only for surfaces, but the proof in higher dimensions is not significantly different. Hopf-Rinow theorem has been studied in detail in both Riemannian and Finslerian geometries in the literature, the best general references here are \cite{BCS00, Do}, \cite{ON}. In the Finsler case forward geodesic completeness is involved, only.

After a development of the sub-Riemannian geometry as well as its generalization, namely sub-Finslerian geometry,  the generalization of core theorems of Riemannian geometry has been started. Relating to our issue, Strichartz \cite{St86}, Rifford \cite{RI} and Agrachev et al. \cite{ABB} gave an extension for a sub-Riemannian case, while Bao et al. \cite{BCS00} showed the Finslerian version of Hopf-Rinow theorem. It turned out that in sub-Riemannian geometry, for general complete sub-Riemannian structures, the exponential mapping is not surjective. This is due to the fact that we may have abnormal minimizing curves and this is the case in the sub-Finslerian context, too.

To prove the statements of Hopf-Rinow theorem in the sub-Finsler setting, we need the following concepts and explanations. First, in Section 2 we review some of the standard facts of sub-Finsler geometry. In the third Section, we extend our discussion about the Legendre transformation (see \cite{L.K}) to define the sub-Finsler manifold on the distribution $\mathcal{D}^*$ of the cotangent space, where we look more closely at a sub-Hamiltonian $H$ defined on $\mathcal{D}^*$, induced by the sub-Finslerian metric $F^*$.
Afterwards, we construct a sub-Finsler bundle, which plays a major role in the formalization of the sub-Hamiltonian in sub-Finsler geometry, in Section 4. Moreover, the sub-Finsler bundle allows an orthonormal frame for the sub-Finsler structure. In Section 5, we introduce the notion of an exponential map in sub-Finsler geometry. In the last section our main theorem is stated and proved.

\section{Definitions and some properties of sub-Finsler manifolds}
\label{def}
In this section we review some of the standard facts on the sub-Finsler metrics and
set up the notations and the terminology which will play an essential role in this paper, for more details we refer the reader to \cite{L.K, L.K1, L.K2}.
\begin{definition} \label{FH}
Let $M$ be an $n$--dimensional connected manifold. A sub-Finslerian structure on $M$ is a triple $(\mathcal{D}, \sigma , F)$ where:
\begin{itemize}
  \item [(1)] $(\mathcal{D}, \pi_{\mathcal{D}})$ is a vector bundle on $M$.
  \item [(2)] $\sigma :\mathcal{D} \to TM$ is a morphism of vector bundles.
  In particular, the following diagram is commutative
  \begin{center}
\begin{minipage}[t]{4cm}
$$\bfig
  \morphism(400,0)|r|<0,-400>[`M;\pi]
   \morphism(0,0)|a|<400,0>[\mathcal{D}`TM;\sigma]
    \morphism(0,0)|b|<400,-400>[`M;\pi_{\mathcal{D}}]
 \efig$$
\end{minipage}
\end{center}
such that $\pi_{\mathcal{D}}: \mathcal{D} \to M$ and $\pi: TM \to M$ are the canonical projections.
  \item [(3)] A function $F: \widetilde{\mathcal{D}}\rightarrow \R$, where
$\widetilde{\mathcal{D}}=\mathcal{D} \setminus \{0\}$, called a {\em sub-Finsler metric}, which satisfies
the following properties:
\\
$\bullet$ ($\mathbf{Positive \ definiteness}$): $F_x(v) > 0$ for all $v\in \widetilde{\mathcal{D}}, \, x\in  M$.\\
 $\bullet$ ($\mathbf{Regularity}$): $F$ is smooth, i.e. $C^{\infty}$ on $\widetilde{\mathcal{D}}$.\\
  $\bullet$ ($\mathbf{Positive \ homogeneity}$): $F_x(\lambda v)=\lambda F_x(v)$ for all $v\in
    \widetilde{\mathcal{D}}_x\  \mbox{and}\ \lambda \in \mathbb{R}^+$.\\
$\bullet$ ($\mathbf{Strong \ convexity \ condition}$): The Hessian matrix of $F^2$ with respect to the coordinates on the fibre is positive definite.

One can replace the strong convexity condition by the following subadditivity property (in an equivalent terminology, a triangle inequality):
$$F_x(v+u) \leqslant F_x(v) + F_x(u),\ \mathrm{for \ all} \ v,u \ \in  \widetilde{\mathcal{D}}.$$
\end{itemize}
    A {\em sub-Finsler manifold} is a smooth manifold $M$ endowed with a sub-Finslerian structure, i.e. the triple $(\mathcal{D}, \sigma, F)$.
\end{definition}

Let $\mathcal{D}_x$ be the fiber over $x\in M$. The last
condition of the sub-Finsler metric means that the matrix $\frac{\partial^2F^2}{\partial
v^i\partial v^j}(x,v)$ is positive definite for all $v=(v^1,\dots,
v^k)\in \mathcal{D}_x$. Equivalently, the corresponding
indicatrix
$$ I_x=\{v \, | \, v\in \mathcal{D}_x, \ F_x(v)=1\}$$ is strictly convex.

The following technique describes the association between the sub-Finsler structure $(\mathcal{D}, \sigma, F)$  and a Finsler metric $\hat{F}$ on  $\mathrm{Im} (\sigma) \subset TM$:

For each $u \in \mathrm{Im} (\sigma)_x \subset T_xM$ and $x \in M$, we have
$$\hat{F}_x(u) = \inf_{v} \{F_x(v) | \ v \in \mathcal{D}_x,  \ \sigma(v)= u \}.$$

From now on we suppose that $\mathcal{D} \subset TM$, \, $\sigma :\mathcal{D} \to TM$ is the inclusion $i :\mathcal{D} \to TM$ and $F$ is a sub-Finsler metric on $\mathcal{D}$.

As in the sub-Riemannian case, we call $\mathcal{D}$ the \textit{horizontal
distribution}. A piecewise smooth curve $\gamma :[0, T]\rightarrow M$ is called
\textit{horizontal}, or \textit{admissible} if $\dot{\gamma} (t)\in
\mathcal{D}_{\gamma(t)}$ for all $t\in [0, T]$, that is, $\gamma(t)$ is
tangent to $\mathcal{D}$. The length of $\gamma$ is
defined as usual by
$$\ell(\gamma)=\int_0^T F(\dot{\gamma}(t))dt.$$

Equivalently, as in the Finslerian case, we observe that it suffices to
minimize the {\it energy}
$$E(\gamma)=\frac{1}{2}\int_0^T F^2(\dot{\gamma}(t))dt.$$
instead of length $\ell(\gamma)$.

The length induces a sub-Finslerian distance
$d(x, y)$ between two points $x$ and $y$ as in Finsler geometry:
$$ d(x, y)=\inf \{\ell(\gamma) \ | \gamma:[0,T] \to M \ \mathrm{horizontal},\ \gamma(0)= x, \gamma(T)= y \},$$
where we consider the infimum over all 
horizontal curves
joining $x$ and $y$. The distance is infinite if there is no such a horizontal
curve between $x$ and $y$. In addition, the horizontal
curve $\gamma :[0, T]\rightarrow M$ is called a {\em length minimizing} (or simply a {\em minimizing}) geodesic, if
it realizes the distance between its end points, that is, $\ell(\gamma) = d(\gamma (0), \gamma (T))$.

Chow theorem answers to the following question: given two points
$x$ and $y$ in a sub-Finsler manifold, is there a horizontal curve
that joins $x$ and $y$?

In the case of an involutive
distribution $\mathcal{D}$ the Frobenius theorem asserts that  the set of
the horizontal paths through $S$ form a smooth immersed submanifold,
the leaf through $x$, of dimension equal to the rank of
distribution $k$. In this case, if $\mathcal{D}$ is involutive and $y$ is
not contained in the leaf through $y$, there is no any horizontal
curve joining $x$ and $y$.

A positive answer is given by the Chow theorem in the case of
bracket generating distributions, which are the "contrary" of the
involutive distributions.

\begin{definition}\cite{Mo02}
A distribution $\mathcal{D}$ is said to be {\em bracket generating} if any local
frame $X_i$ of $\mathcal{D}$, together with all of its iterated Lie
brackets spans the whole tangent bundle $TM$.
\end{definition}

\begin{theorem}(Chow's theorem \cite{Mo02})
If $\mathcal{D}$ is a bracket generating distribution on a connected
manifold $M$ then any two points of $M$ can be joined by a
horizontal path.
\end{theorem}
\begin{remark}
  The problem of minimizing the length of a curve joining two given points $x$ and $y$ is equivalent to a time optimal problem:
where the control bundle is $(\mathcal{D}, \pi_{\mathcal{D}} , M)$ and we are searching for such a curve $\gamma(t)$ and a control curve $v(t) \in \mathcal{D}_{\gamma(t)}$
minimizing the time $T$ needed to connect $x$ and $y$.
\end{remark}

\section{Legendre transformation of sub-Finslerian geometry}
Let $\mathcal{D}^*$ be a distribution of rank $s$ on a smooth manifold $M$ that assigns to each point $x \in U \subset M$ a linear subspace  $\mathcal{D}^*_x \subset T^*_xM$ of dimension $s$, see \cite{L.K2}. In other words, $\mathcal{D}^*$ of rank $s$  is a smooth subbundle of rank $s$ of the cotangent bundle $T^*M$. Such a field of cotangent $s$-planes is spanned locally by $s$ pointwise linear independent smooth differential $1$-forms, namely, $$\mathcal{D}_x^*= \textrm{span} \{\alpha_{1}(x), \ldots, \alpha_{s}(x)\}, \qquad \alpha_i(x) \in \mathfrak{X}^*(M).$$
In addition, we refer to $\mathcal{D}^0_x $ as the annihilator of the distribution $\mathcal{D}$ (isomorphic to $\mathcal{D}$), of rank $n-k$, which is the set of all covectors that annihilates the vectors in $\mathcal{D}_x$, i.e.
\begin{equation}\label{ann1}
  \mathcal{D}^0_x = \{\alpha \in T^*_xM : \alpha(v)= 0 \ \forall \ v \in \mathcal{D}_x\}.
\end{equation}

In \cite{L.K}, we introduced the Legendre transformation of sub-Finsler geometry. Let us briefly recall it:

The {\it sub-Lagrange function} $L: \mathcal{D} \to \mathbb{R},$ determined by $F$ is given in the following way:
$L= \frac{1}{2}F^2.$
The fiber derivative of $L$ defines the map
 $$\mathcal{L}_L: \mathcal{D}  \to\mathcal{D}^*,\qquad
 \mathcal{L}_L(v)(w)= \frac{d}{dt} L_x(v+ tw), \ \text{where} \ v, w \in \mathcal{D}_x,$$
called the {\it Legendre transformation} of $(M, \mathcal{D}, F).$

 We denote by $(x^i)$ the coordinate in a neighborhood $U \subset M$ with $(x^i, v^a)$ in $\mathcal{D}|_U \subset TM$, and $(x^i, p_a)$ in $\mathcal{D}^*|_U \subset T^*M$, respectively, where $i=1, \dots , n, \ a=1, \dots , k$. Then the relation of the distribution $\mathcal{D}$ of the tangent bundle and the distribution $\mathcal{D}^* $ of the cotangent bundle is given by the Legendre transformation in local coordinates as follows
$$\mathcal{L}_L(x^i, v^a)= (x^i,  \frac{\partial L}{\partial v^a}).$$
Then the {\it sub-Hamiltonian} is given by
$$H: \mathcal{D}^* \to \mathbb{R},$$
$$ H=  \iota_{\mathcal{L}^{-1}_L} - L\circ \mathcal{L}^{-1}_L,$$
where $\iota_{v} (p)= \langle v, p\rangle = p(v) $ for any $v = \mathcal{L}^{-1}_{L}(p) \in \mathcal{D}$ and $p \in \mathcal{D}^*$. Moreover, locally given by,
 $$ H(x^i, p_a) = v^ap_a - L(x^i, v^a), \ \mbox{where} \ p_a = \frac{\partial L}{\partial v^a}.$$
Secondly, using the fiber derivative of $H$, we define the Legendre transformation
of the sub-Hamiltonian $H$ in the following way:
$$\mathcal{L}_{H}: \mathcal{D}^* \to\mathcal{D},$$
For any $p, q \in \mathcal{D}^*_x$, it holds
$$q(\mathcal{L}_{H}(p))= \frac{d}{dt} H(x, p+ t q).$$
This locally relates the distribution $\mathcal{D}^* $ of the cotangent bundle and the distribution $\mathcal{D}$ of the tangent bundle according to the next expression:
$$\mathcal{L}_{H}(x^i, p_a)= (x^i,  \frac{\partial H}{\partial p_a}).$$
Naturally,
$\mathcal{L}_{L}$ and $\mathcal{L}_{H}$ are inverses of each other:
\begin{align*}
 \mathcal{L}_{H} \circ \mathcal{L}_L&= 1_{\mathcal{D}},  &
    \mathcal{L}_L \circ \mathcal{L}_{H} &= 1_{\mathcal{D}^*}.
\end{align*}

In other hand, for every $p \in  \mathcal{D}^*_x$, one can define the sub-Finsler metric $F^* \in \widetilde{\mathcal{D}}^* \sim T^*M\setminus\mathcal{D}^0$ with help of the indicatrix $I_x$ as follows:
\begin{equation*}
F_x^*(p):=  \sup_{\substack{w \in I_x}} \ p(w) = \sup_{\substack{0 \neq v \in \mathcal{D}_x}} \ p [\frac {v}  {F_x(v)}].
\end{equation*}

Observed that $\widetilde{\mathcal{D}}^*$ is the subbundle of the cotangent bundle obtained by removing the zero cotangent vector from each fibre. In fact, $F^*$ turns out to meet the same properties that mentioned in Definition \ref{FH}, but on $\mathcal{D}^*$  instead of $\mathcal{D}$. Then

$$F^*(p)= F(v), \ \mathrm{where}\ p =\mathcal{L}_L(v),\quad\mbox{and}\quad
 H := \frac {1}{2} (F^*)^2,$$
see details in \cite{BCS00}.

\section{Sub-Finsler bundle}

We define in this section a sub-Finsler vector bundle which will play a major role in the formalization of the sub-Hamiltonian in sub-Finsler geometry. Let us consider first the covector subbundle $ (\mathcal{D}^*, \tau, M)$ with the projection ${\tau} : \mathcal{D}^* \to M$, which is a subbundle of rank $k$ (= dim $\mathcal{D}^*$) in the cotangent bundle of $T^*M$.
The illustrious role in our consideration will play by the pullback bundle
 $\tau^*({\tau})= (\mathcal{D}^* \times \mathcal{D}^*, \mathrm{pr}_1, \mathcal{D}^*)$ of ${\tau}$ by $\tau$ as follows:
   $$\mathcal{D}^* \times_{M} \mathcal{D}^* := \{(p, q) \in  \mathcal{D}^* \times \mathcal{D}^* | \ \ \tau(p)= {\tau}(q) \},$$
$$\mathrm{pr}_1 : \mathcal{D}^* \times_{M} \mathcal{D}^* \to \mathcal{D}^*, \ (p, q) \mapsto p.$$

Throughout, we call the above pullback bundle as the {\em sub-Finsler bundle} over $\mathcal{D}^*$.
Now, if $p$ is fixed, then
\begin{align*}
  (\mathrm{pr}_1)^{-1}(p) & = \{ (p, q) \in  \mathcal{D}^* \times \mathcal{D}^*  | \ \ q \in \mathcal{D}^*_{\tau(q)} \} \\
   & = \{p\} \times \mathcal{D}^*_{\tau(p)},
\end{align*}
is a fiber of the sub-Finsler bundle over $p \in \mathcal{D}^*$.

We can introduce a Riemannian metric $g^*$ on the sub-Finsler vector bundle induced by the sub-Hamiltonian $H$ as follows:
$$
 \langle  q, r\rangle_p = g^*_p (q,r) := \frac{\partial^2 H(p+tq+sr)}{\partial t \partial s}{|_{t,s=0}} \qquad \mbox{for all}\ q,r\in \mathcal{D}^*_{{\tau}(p)},
$$
which locally means
$$ g^{*ij}= \frac{\partial^2 H}{\partial p_i\partial p_j}.
$$

Now the sub-Finsler bundle $\tau^*({\tau})$ allows $k$ covector fields $X_1, X_2, \dots, X_k$  which form an orthonormal frame with respect to the induced Riemannian metric $g^*$.

Notice that $X_i(p)$ is a covector field that depends on the position $x \in M$ and the direction $p \in \mathcal{D}^*$. Moreover, one can choose in a way that $X_i(p)$ is a homogeneous of degree zero in $p$,
i.e. $X_i(tp) = t^0X_i(p)= X_i(p).$
According to the above metric ${g^*}^{ij}$ on $M$ which is homogeneous of degree zero, we could generate a new formalism of the sub-Hamiltonian function in the components $p_i$ (induces naturally by the inner product, see \cite{CD})
\begin{equation}\label{M}
  H(x, p) =  \frac{1}{2} \sum_{i,j=1}^{n}  {g^*}^{ij} p_i p_j,
\end{equation}
such that this metric defined in the extended Finsler metric which was shown in \cite{L.K}.
We can write the sub-Hamiltonian function (\ref{M}) in a more useful way using the orthonormality of $X_i$ as follows
\begin{equation}\label{H}
  H(x, p) =  \frac{1}{2} \sum_{i=1}^{k}  \langle  p, X_i(p)\rangle^2,\qquad p\in \mathcal{D}^*_x.
\end{equation}
One can easily check the homogeneity of degree 2 in $p$ of the sub-Hamiltonian function $H(x, p)$:
\begin{equation}\label{HOM}
     H(x, tp)   = \frac{1}{2} \sum_{i=1}^{k}  \langle  tp, X_i(tp)\rangle^2
        = \frac{t^2}{2} \sum_{i=1}^{k}  \langle  p, X_i(p)\rangle^2
= t^2H(x, p).
\end{equation}

The importance of  $H(x, p)$ is to define sub-Finslerian geodesics.
Our function $H(x, p)$ produces a system of sub-Hamiltonian differential equations,
since it is a smooth function on $\mathcal{D}^*$. Such differential equations are in terms of canonical coordinates $(x^i, p_i)$.
\begin{definition} \label{NO}
The generated sub-Hamiltonian differential equations
\begin{align*}
\dot{x}^i & = \frac{\partial H}{\partial p_i}(x, p), \\
\dot{p}_i & = - \frac{\partial H}{\partial x^i}(x, p), \quad i= 1, \ldots, n,
\end{align*}
are called {\it normal geodesic equations}.
 \end{definition}
 \begin{lemma} \label{SOl1}
  If $\xi(t):=(x(t), p(t))$ is a solution of the sub-Hamiltonian system for all $t \in \mathbb{R}$, then there exists a constant $c \in \mathbb{R}$ such that $H(x(t), p(t))=c$.
\end{lemma}
\begin{proof}
 Taking the derivative of  $H(x(t), p(t))$ w.r.t. $t$, we get
  $$\frac{d}{d t} H(x(t), p(t)) = \frac{\partial H}{\partial x^{i}}(x(t), p(t)) \dot{x}(t)+\frac{\partial H}{\partial p_{i}}(x(t), p(t)) \dot{p}(t).$$
  Replacing $\dot{x}(t)$ and $\dot{p}(t)$ by the above sub-Hamiltonian differential equations in the Definition \ref{NO}, we obtain
 \begin{align*}
      \frac{d}{d t} H(x(t), p(t)))  & =\frac{\partial H}{\partial x^{i}}(x(t), p(t)) \frac{\partial H}{\partial p_{i}}(x(t), p(t))-\frac{\partial H}{\partial p_{i}}(x(t), p(t)) \frac{\partial H}{\partial x^{i}}(x(t), p(t)) \\
    & =0.
 \end{align*}
Therefore $H(x(t), p(t))$ is constant.
\end{proof}
\begin{remark} \label{Ca2}
From Lemma \ref{SOl1}, it follows that any solution $\xi(t):=(x(t), p(t))$  of the sub-Hamiltonian differential equations on $\mathcal{D}^*$
for a sub-Hamiltonian function $H(p)$ satisfies  $H(x(t), p(t))=c$.
Let the projection $x (t)= \tau(\xi(t)) \in M$, so each sufficiently short subarc of $x(t)$ is a minimizer sub-Finslerian geodesic, {\rm  (see \cite[Corollary 2.2]{CR})}. In addition, this subarc is the unique minimizer joining its end points.
\end{remark}
The projection curve $x(t)$ mentioned above is said to be the {\em normal sub-Finslerian geodesics}
or simply the {\em normal geodesics}.

\begin{remark}
In the sub-Finslerian geometry, not all the sub-Finslerian geodesics are normal (contrary to the Finsler geometry). This is due to the fact that the sub-Finslerian geodesics which are also a minimizing geodesic might not be solved the sub-Hamiltonian system. Those minimizer that are not normal geodesics called {\it singular} or {\it abnormal} geodesics (see \cite{Mo02} for more details).
\end{remark}

Moreover, we call the extremal pair $\xi(t) =(x(t), p(t))$ a {\em normal extremal} if it is a solution for the sub-Hamiltonian system, otherwise it is called  an {\em abnormal extremal}.

Turning to the relationship between the normal geodesic and the locally length-minimizing horizontal curves, Calin et al. proved in \cite{CD} that any normal geodesic is a horizontal curve and a locally length-minimizing horizontal curve.
After all, by using  (\ref{H}) one can generate the  system of differential equations in terms of canonical coordinates $(x, p)$ as follows:
\begin{align}\label{HH}
\dot{x}^i & = \frac{\partial H}{\partial p_i} = \sum_{j=1}^{k}  \langle  p, X_j(p)\rangle \ (\delta_i (X_j(p)) + \langle p, D_{p_i} X_j (p) \rangle ), \\
\dot{p}_i & = - \frac{\partial H}{\partial x^i} = - \sum_{j=1}^{k}  \langle  p, X_j(p)\rangle \langle  p, D_{x^i}X_j(p)\rangle,
\end{align}
 where $\delta_i$ is the $i$-th coordinate function.

\section{Exponential map in sub-Finsler geometry}
Let $(M,  d)$ be a general metric space, such that $M$ is an $n$-dimensional manifold and the function $d:M \times M \to \mathbb{R}^+ \cup \{\infty\}$,
 is called  a  metric if have the  following properties:  for  all  $x,y,z  \in   M$,
\begin{itemize}
  \item [(i)] $d(x, y) = 0$, with equality if  and  only  if $ x = y$;
  \item [(ii)] $d(x,  y)  +  d(y,  z)  \leq  d(x,  z).$
\end{itemize}
 If the function $d$ is an asymmetric, then we can define the forward metric balls and forward metric spheres, with center $x \in M$ and radius $ r \ \textgreater \ 0$ as follows:
$$ B_ x (r)= \{ \ y \in M: \  d(x, y) \ \textless \ r \},$$
$$ S_ x (r)= \{ \ y \in M: \ d(x, y) =  r \}. $$
The cotangent balls and the cotangent spheres in $\mathcal{D}^*$ are defined as follows:
$$ \mathcal{B}^*_ x (r)= \{ \ p \in  \mathcal{D}^*: \  F^*_x(p) \ \textless \ r \},$$
$$ \mathcal{S}^*_ x (r)= \{ \ p \in  \mathcal{D}^*: \ F^*_x(p) =  r \},$$
 for any fix $x \in M$ and radius $ r$.

A  subset  $U  \subset M$ is  said  to  be  open  if,  for  each  point  $x  \in  U$,  there  is
a   forward metric ball  about  $x$  contained  in  $U$.  Then  we  get  the  topology  on
$M$ and  all  metric  spaces  are  first  countable  and  $T_1$-spaces.  In  general,
we  assume  that  the  metric  $d$  of  any  metric  space  $(M,  d)$  is  continuous  with
respect  to  the  product  topology  on  $M\times M$.  Thus,  every  backward  metric ball, i.e.  $  B^-_ x (r)= \{ \ y \in M: \  d(y, x) \ \textless \ r \}$, is
open  and  the  metric  space  is  a  Hausdorff  ($T_2$)  space.  Hence  the  compact
sets  in  such  a  space  are  closed.

As a result of the above, we  immediately  have  the  following

\begin{proposition}\label{MCo}
  In  a  metric  space  $(M,  d)$  the  following  are  equivalent:
  \begin{itemize}
    \item [(i)] 	A  sequence  $\{x_k \}$  in  $(M,  d)$  converges  to  $x \in M$  in  the  sense  of  topology.
    \item [(ii)] 	$\lim_{k\rightarrow \infty} d(x,  x_k)  =  0$.
  \end{itemize}
\end{proposition}

\begin{proposition} \label{EX}
Let $x$ be any point in a (reversible) sub-Finslerian manifold $M$,
and $\bar{B}_x(r)$ is a compact ball, for some $r > 0$.
Then for any $y \in B_x(r)$ there is a minimizing geodesic from $x$ to $y$, that is,
  $$d(x,y)= \mathrm{min}\{\ell(\gamma) \ | \gamma:[0,T] \to M \ \mathrm{horizontal},\ \gamma(0)= x, \gamma(T)= y \}.$$
\end{proposition}
\begin{proof}
  Fix $y \in B_x(r)$ and suppose that $\gamma_k :[0,T] \to M$ is a minimizing sequence of horizontal paths with unit speed from $x$ to $y$ and such that
  $$\lim_{k\rightarrow \infty} \gamma_k (0)= x, \quad \lim_{k\rightarrow \infty} \gamma_k (T)= y, \quad  \lim_{k\rightarrow \infty} \ell(\gamma_k)= d(x,y).$$
  For the reason that $d(x,y) < r$, we get $\ell(\gamma_k) \leq r$ for all $k \geq k_0$ large enough.
 Proposition \ref{MCo} asserts that the metric $d$ is continuous under the topology of the manifold and the reversibility of
$F$ holds on a compact set.
 Consequently, any sequence $\gamma_k$ of curves which have uniformly bounded lengths has an uniformly convergent subsequence (Ascoli–-Arzela theorem), we denote this subsequence by the same symbol, and a Lipschitz curve $\gamma: [0,T] \to M$.

From above one can assume that $\gamma_k: [0,T] \to M$ is a convergent subsequence of length minimizers
parametrized by arc length (i.e. $F(\dot{\gamma}(t))=1$)
on $M$ such that such that $\gamma_k \to \gamma$ uniformly on $[0, T]$. This gives that
$$\ell(\gamma_k)  = d(\gamma_k (0),  \gamma_k (T)),$$
which is due to the claim that $\gamma_k$ is a minimizing geodesic.
   The sequence $\gamma_k$ converges uniformly if for every $\epsilon > 0$  there is
   a natural number ${\displaystyle N}$ such that for all ${\displaystyle n\geq N}$
   and all $t \in [0, T]$  one has  $d(\gamma_k (t),  \gamma(t)) < \epsilon.$
   Further, the semicontinuity of the length implies that if $ \lim_{k\rightarrow \infty} \gamma_k = \gamma$ then
   $$\ell(\gamma) \leq \lim_{k\rightarrow \infty} \mathrm{inf}\  \ell(\gamma_k).$$
   Now, by continuity of the distance, we obtain
   $$\ell(\gamma) \leq \lim_{k\rightarrow \infty} \mathrm{inf} \ \ell(\gamma_k) = \lim_{k\rightarrow \infty} \mathrm{inf} \ d(\gamma_k (0),  \gamma_k (T))= d(\gamma (0),  \gamma (T)).$$
   This yields that $\gamma$ is minimizing geodesic, i.e. $\ell(\gamma)= d(x,y)$.
   The horizontal property of $\gamma$ follows in the same way as was done in \cite{ABB}, Theorem 3.41.
   \end{proof}
Next, we define the exponential map.  For the general case, roughly speaking, if $M$ is a smooth Finsler manifold, $x$ a point in $M$
and $u \in T_xM$. Then the exponential map is given by
$$\mathrm{exp}_x: T_xM \to M,$$
such that $\mathrm{exp}_x(u)= \gamma_{u}(1)$ for the unique geodesic $\gamma$ that starts at $x$ and has initial speed vector $u$.
Furthermore, in the dual space the exponential map for every $x \in M$ and $p \in T^*_xM$ defined by
$$\mathrm{exp}^*_x: T^*_xM \to M,$$
such that $\mathrm{exp}^*_x(p)= \gamma_{p}(1)$ for the unique geodesic $\gamma$ that starts at $x$ and has initial speed vector $u=\mathcal{L}^{-1}_L(p)$,
where $L$ here is the  Lagrangian of the Finsler manifold.

The exponential map is an essential object in sub-Finslerian geometry, parametrizing normal extremals through their initial covectors. We are going to define the exponential map in both of the distribution
 $\mathcal{D}, \mathcal{D}^*$ of the tangent and the cotangent bundles respectively.

\begin{definition}
Let $\Omega_x\subset  \mathcal{D}_x$ be the domain of the exponential map over $x\in M$ such that $\Omega_x$ given by
$$ \Omega_x = \left \{ v \in \mathcal{D}_x | \ \xi \  \text{is  defined   on  the  interval}\ [0,1]\right \},$$
where $v =\mathcal{L}_H(p)$ by the Legendre transformation of sub-Hamiltonian $H$, and $ \xi(t)$ is the normal extremal.
Then the {\em sub-Finsler exponential map} is defined as follows
$$\mathrm{exp}_x:\Omega_x \subset  \mathcal{D}_x  \subset T_xM\ \to \ M,  \  v \mapsto \pi_{\mathcal{D}}(\mathcal{L}_H( \xi(1))).$$

We can do the same in the distribution $\mathcal{D}^*_x$. Let $\Omega^*_x\subset  \mathcal{D}^*_x$ be the domain of the exponential map over $x\in M$ such that $\Omega^*_x$ given by
$$ \Omega^*_x = \left \{ p \in \mathcal{D}^*_x | \ \xi \  \text{is  defined   on  the  interval}\ [0,1]\right \}.$$
Consequently, the {\em sub-Hamiltonian exponential map} is given by
$$\mathrm{exp}^*_x:\Omega^*_x \subset  \mathcal{D}^*_x   \subset T^*_xM\ \to \ M,  \  p \mapsto \tau( \xi(1)),$$
where $\xi(t)$ is the same normal extremal as above. The set $ \Omega^*_x $ contains the origin and star-shaped with respect to 0.
Moreover, with the help of Legendre transformation it is fairly easy to see that
\begin{equation}\label{EXPO}
\mathrm{exp}_x(v)= \mathrm{exp}^*_x(p), \quad \mathrm{where} \quad p =\mathcal{L}_L(v).
\end{equation}
It follows that the normal sub-Finslerian
geodesics $x(t) = \tau(\xi(t))$ satisfies
$$x(t) = \mathrm{exp}^*_x(tp), \quad \mathrm{for \ all} \ t \in [0,T].$$
\end{definition}

\begin{theorem}
  The exponential  mapping $\mathrm{exp}^*_x$  is  a  local  diffeomorphism on $\mathcal{D}^*_x \subset T_x^*M\backslash\{0\}$.
\end{theorem}
\begin{proof}
  In \ref{HOM}, we show the homogeneity of the sub-Hamiltonian function $H(x, p)$ with  respect  to  $p$.
   So, for  any  constant  $a  > 0$,  the curve  $\xi(at)  :  (\epsilon/a,\epsilon/a)  \to   M$  is  the  same  geodesic  satisfying  the  initial
conditions   $\tau(\xi_p(0))   =  x$  and   $\xi_p(0)   =  ap$,  i.e.,
$$\tau(\xi_p(at))   =  \tau(\xi_{ap}(t)).$$
Since the sub-Hamiltonian vector field $$\vec{H}(x, p) = g^{ab}(x, p)p_b  \frac{ \partial}{\partial x^a} - \frac{1}{2} \frac{\partial g^{ab}}{\partial x^k}(x, p) p_ap_b \frac{ \partial}{\partial p_k},$$ that introduced in \cite{L.K}, is smooth except  for $p=0$ where  it  is $C^1$. Then $\mathrm{exp}^*_x$ is $C^{\infty}$ on $\mathcal{D}^*_x \subset T_x^*M\backslash\{0\}$, while  it  is  $C^1$  at  $p  =  0$  and  $d(\mathrm{exp}^*_x)|_{0}  =  \mathrm{id}$.  Thus,  $\mathrm{exp}^*_x$  is  a  local  diffeomorphism.
\end{proof}
By  equation (\ref{EXPO}),  one  can  get  the  following

\begin{corollary}
  The sub-Finsler exponential map $\mathrm{exp}_x$ is a $C^{\infty}$ away from the zero section of $\mathcal{D}$ and only $C^1$ at the zero section such that for each $x \in M$
$$d(\mathrm{exp}_x)_{|0} : \Omega_x \subset  \mathcal{D}_x \to \Omega_x \subset  \mathcal{D}_x,$$
is the identity map at the origin $0 \in \mathcal{D}_x$.
\end{corollary}

\begin{remark} \label{exp}
 It is clear that in the case of sub-Finsler exponential map the following expressions holds:
$$\mathrm{exp}^*_x[ \mathcal{B}^*_ x (r)]= B_ x (r),$$
$$\mathrm{exp}^*_x[ \mathcal{S}^*_ x (r)]= S_ x (r),$$
which are analogous to the Finslerian context, see Bao et al.\ \cite{BCS00} for more details.
\end{remark}

\begin{remark}
Turning to sub-Riemannian case, Strichartz in \cite{St86} stated that for bracket generating distributions the exponential map is a local diffeomorphism. This is due to the fact that the solutions of the
sub-Hamiltonian system depend differentially on the initial data. But this is a difference from the
Riemannian context, the exponential map is not a diffeomorphism at the origin just like the Finslerian case.
\end{remark}

\section{Hopf-Rinow Theorem in sub-Finslerian geometry}

In the following, one can see the explanation of the terms that will be used in Hopf-Rinow Theorem. A sub-Finsler manifold is
said to be {\em forward complete} if every forward Cauchy sequence converges, and it is a {\em forward geodesically complete} if every  normal geodesic $\gamma(t), t \in [0, T) $ parametrized to have unit speed, can be extended to a geodesic for all $ t \in [0, \infty).$ A subset is said to be {\em forward bounded} if it is contained in some forward metric ball $B_x (r)$.
\begin{theorem} Let $(M, \mathcal{D},F)$ be any connected sub-Finsler manifold, where $\mathcal{D}$ is bracket generating distribution. The following conditions are equivalent:
\begin{itemize}
\item [(i)] The metric space $(M, d)$ is forward complete.
\item [(ii)] The sub-Finsler manifold $(M, \mathcal{D},F)$ is forward geodesically complete.
\item [(iii)] $\Omega^*_x =  \mathcal{D}^*_x$, additionally, the exponential map is onto if there are no strictly abnormal minimizer.
\item [(iv)] Every closed and forward bounded subset of  $(M, d)$  is compact.

 \end{itemize}
Furthermore, for any $x, y \in M$ there exists a minimizing geodesic $\gamma$ joining $x$ to $y$, i.e. the length of this geodesic is equal to the distance between these points.

\begin{proof}
(i)$ \implies$ (ii)  Let $\gamma(t) :[0, T)\to\ M$ be a unit speed and maximally forward extended geodesic, $t \in [0, T) $. If we assume that $T\neq \infty$, and choose a sequence $\{t_i\} \to T$ in $[0, T)$ then ${\gamma(t_i)}$ is forward Cauchy, since $$d(\gamma(t_i), \gamma (t_j)) \leq |t_j - t_i| ,   \  \mathrm{for \ all} \  i  \leq  j.$$
Now, (i) makes it obvious that $\gamma(t_i)$ converges to $y \in M$. On one hand, let us define $\gamma (T)$ to be $y$.
On the other hand, Lemma 4.1 in \cite{St86} told us that $\gamma (t)$ can be extended beyond $t=T$. This contradicts our assumption the fact that $T\neq \infty$. Thus, $T=\infty$ for sure, so we have the forward geodesically completeness.

(ii)$ \implies$ (iii) It is sufficient (for first part $\Omega^*_x =  \mathcal{D}^*_x$) to prove that any normal extremal pair $\xi(t)$, starting from the initial conditions, is defined for all $t \in \mathbb{R}$. Suppose that the normal extremal is not extendable to the some interval $[0, T+\delta)$ for all $\delta > 0$ and suppose that it is defined on $[0, T)$. Let $\{t_i\}$ be any increasing sequence such that the limit of this sequence is $T$. Hence, the projection $x (t)= \tau(\xi(t))$ is a curve with unit speed defined on $[0,T)$, therefore, the sequence $\{t_i\}$ is a forward Cauchy sequence on $M$, since
$$d(x(t_i), x(t_j)) \leq | t_i - t_j|.$$

By completeness, it follows that the sequence $x(t_i)$  converges to some point $y \in M$.
We suppose there are coordinates around the point $y$ and an orthonormal frame $X_1,X_2,...,X_k$ in small ball $\mathcal{B}^*_y(r)$ in the sub-Finsler bundle. Let us show that in the
coordinates  $\xi(t)= (x(t),p(t))$ the curve $x(t)$ is uniformly bounded.
This grants a contradiction that the normal extremal is not extendable.
In fact, for every  $p \in  \mathcal{D}^*$, we consider the following non-negative form (\ref{H}) of the sub-Hamiltonian function $H$:
\begin{equation*}\
  H(x, p) =  \frac{1}{2} \sum_{i=1}^{k}  \langle  p, X_i(p)\rangle^2.
\end{equation*}
Then, the sub-Hamiltonian system has the form:
\begin{align*}\
\dot{x}^i(t) & = \frac{\partial H}{\partial p_i} (x(t),p(t)) = \sum_{j=1}^{k}  \langle  p(t), X_j(p(t))\rangle  (\delta_i (X_j(p)) + \langle p, D_{p_i} X_j (p) \rangle ), \\
\dot{p}_i(t) & = - \frac{\partial H}{\partial x^i}(x(t),p(t)) = - \sum_{j=1}^{k}  \langle  p(t), X_j(p(t))\rangle \langle  p(t), D_{x^i}X_j(p(t))\rangle,
\end{align*}
for $t \in [T -\delta, T)$ with $\delta>0$ small enough. Since $D_{\gamma(t)}X_i$ are given in a compact small ball $\bar{\mathcal{B}}^*_y(r)$, they are bounded, so there is a constant $\mathcal{C}>0$ such that
$$|\dot{p}(t)| \leq \mathcal{C} |p(t)| \quad \forall t \in [T -\delta, T).$$
If we apply Gronwall's Lemma (see \cite{RI}, p.122), it leads us to that $|p(t)|$ is uniformly bounded on a bounded interval. This contradicts our assumption that the normal extremal can not be extended beyond $T$.

(iii)$ \implies$ (iv)  Assume that $\bar{A}$ is a closed and forward bounded subset of $(M, d)$. Applying the bracket generating assumption, for every $y \in \bar{A}$, Proposition \ref{EX} asserts that there is a minimizing geodesic $\mathrm{exp}^*_x(tp_y),$ $ 0 \leq t \leq T,$ from $x$ to $y$. The set of all $p_y$ is subset $A$ of $\mathcal{D}^*_x$.
Since $F^*_x(p_y) = d(x,y)$,  and  $d(x,y)\leq r$  for some $r$ due to the forward boundedness of $\bar{A}$,
the subset $A$ is bounded and contained in the compact set $\mathcal{B}^*_x(r) \cup \mathcal{S}^*_x(r)$.
By Remark \ref{exp}, $\mathrm{exp}^*_x[\mathcal{B}^*_x(r) \cup \mathcal{S}^*_x(r)]$ is compact and contained in the closed set $\bar{A}$, then $\bar{A}$ it must be compact.

(iv)$\implies$ (i) Let $\{x_i\}$ be a forward Cauchy sequence in $M$, and by the subadditivity it must be forward bounded.
Choose $A:=\{x_i | i \in \mathbb{N}\}$, then its closure $\bar{A}$ is still forward bounded under the manifold topology of $M$.
Taking into account the assumption (iv), $\bar{A}$ should satisfy the compactness property, therefore, the sequence $\{x_i\}$ contains a convergent subsequence.

Let $\{x_k\}$ be a convergent subsequence, consider it converges to some $y \in \bar{A} \subset M$. In other hand, we need to check that $\{x_i\}$ converges to $y \in \bar{A} \subset M$. To do this, fix $\epsilon \ \textgreater \ 0$, since $\{x_i\}$ is forward Cauchy, there exist a positive number $n_0$ such that $j  \ \textgreater \ i \geq n_0$, then $$ d(x_i,x_j) \ \textless \ \frac{\epsilon}{2}.$$ At the same time $\{x_k\}$ converge to $y$. So there is a positive number $n_1$ such that if $ k \geq n_1$, then $$ d(x_k, y) \ \textless \ \frac{\epsilon}{2}.$$

One can assume that $n$ is greater than $n_0$ and $n_1$. If needed, by expanding $n$ further, there is no loss of generality in assuming that $n$ indeed equals some index of the convergent subsequence. Then $d(x_n, y) \leq \ \displaystyle\frac{\epsilon}{2}$, so, for $i > n$, we get
$$ d(x_i, y) \ \leq \ d(x_i, x_n) + d(x_n, y)  \textless \ \frac{\epsilon}{2} + \frac{\epsilon}{2}=\epsilon.$$
So, we have been shown that every forward Cauchy sequence is convergent. Hence $(M,d)$ is forward complete.

At the end, we can use the same proof of Proposition \ref{EX} to verify that for every $x, y \in  M$ there exists a length minimizing geodesic joining $x$ and $ y$, and it has to be  normal geodesic by Remark \ref{Ca2}. Finally, the property of compactness and completeness with help of Proposition \ref{EX}, proves the second part of (iii).

\end{proof}

\end{theorem}

\end{document}